\documentclass{amsart}

\usepackage[latin1]{inputenc}
\usepackage{amsfonts}
\usepackage{amsmath}
\usepackage{amsthm}
\usepackage{amssymb}
\usepackage{latexsym}
\usepackage{enumerate}

\newtheorem{theorem}{Theorem}[section]
\newtheorem{lemma}[theorem]{Lemma}
\newtheorem{proposition}[theorem]{Proposition}

\newtheorem{definition}[theorem]{Definition}

\numberwithin{equation}{section}

\begin{document}

\newcommand{\cc}{\mathfrak{c}}
\newcommand{\N}{\mathbb{N}}
\newcommand{\Q}{\mathbb{Q}}
\newcommand{\R}{\mathbb{R}}

\newcommand{\PP}{\mathbb{P}}
\newcommand{\forces}{\Vdash}
\newcommand{\dom}{\text{dom}}
\newcommand{\osc}{\text{osc}}

\title [Isometrically Universal] 
{An isometrically universal Banach space induced by a non-universal Boolean algebra}

\author{Christina Brech}
\thanks{The first author was partially supported by FAPESP grant (2012/24463-7) and by CNPq grant (307942/2012-0).} 
\address{Departamento de Matem\'atica, Instituto de Matem\'atica e Estat\'\i stica, Universidade de S\~ao Paulo,
Caixa Postal 66281, 05314-970, S\~ao Paulo, Brazil}
\email{brech@ime.usp.br}

\author{Piotr Koszmider}
\address{Institute of Mathematics, Polish Academy of Sciences,
ul. \'Sniadeckich 8,  00-656 Warszawa, Poland}
\email{\texttt{piotr.koszmider@impan.pl}}
\thanks{The research of the second author was partially supported by   grant
PVE Ci\^encia sem Fronteiras - CNPq (406239/2013-4).} 


\begin{abstract} 
Given a Boolean algebra $A$, we construct another Boolean algebra $B$ with no uncountable
well-ordered chains such that the Banach space
of real valued continuous functions $C(K_A)$ embeds isometrically into $C(K_B)$, 
where $K_A$ and $K_B$ are the Stone spaces of $A$ and $B$ respectively.
As a consequence we obtain the following: 
If there exists an isometrically universal Banach space for the class of Banach spaces of a given uncountable density $\kappa$, then there is  another
such space which
is induced by a Boolean algebra which is not universal for Boolean algebras 
of cardinality $\kappa$. Such a phenomenon cannot happen on the level of separable Banach
spaces and countable Boolean algebras.
This is related to the open question if the existence of an isometrically 
universal Banach space and of a universal Boolean algebra are equivalent
on the nonseparable level (both are true on the separable level).
\end{abstract}

\maketitle
\section{Introduction}

If $A$ is a Boolean algebra, we denote by $K_A$ the Stone space of $A$,
that is a  compact Hausdorff totally disconnected 
space such that $A$ is isomorphic to the algebra of all clopen subsets of $K_A$ (see \cite{koppelberg}).
 $C(K)$ denotes the Banach space of real-valued continuous functions on a compact Hausdorff space $K$
with the supremum norm.

Given an infinite cardinal $\kappa$, let $\mathcal B_\kappa$ denote the class of Banach spaces of density at most
 $\kappa$, $\mathcal C_\kappa$ denote
the class of compact spaces of weight at most $\kappa$ and
$\mathcal A_\kappa$ denote the class of Boolean algebras of cardinality at most $\kappa$.
We say that $X\in \mathcal B_\kappa$ is isometrically universal for $\mathcal B_\kappa$ 
if for every $Y\in \mathcal B_\kappa$  there is a linear isometry $T: Y\rightarrow X$ onto its range. We say that $K\in \mathcal C_\kappa$ is universal for $\mathcal C_\kappa$
if for every $L\in \mathcal C_\kappa$ there is a continuous surjection $\phi: K\rightarrow L$ and we say that $ A\in \mathcal A_\kappa$ is universal for $\mathcal A_\kappa$ if for every $B\in \mathcal A_\kappa$
there is a Boolean isomorphism $h: B\rightarrow A$ onto its range.

Classical functorial arguments involving the dual ball of a Banach space imply (see e.g. the introduction to
\cite{universal} or \cite{unisur}) that the first two statements below 
are equivalent (by the Stone duality) and that they imply the third one:
\begin{itemize}
\item There is a universal Boolean algebra for $\mathcal A_{2^\omega}$.
\item There is a universal compact Hausdorff space  for $\mathcal C_{2^\omega}$.
\item There is an isometrically universal Banach space for $\mathcal B_{2^\omega}$.
\end{itemize}
 By classical results (e.g. the Banach-Mazur Theorem) the three statements are true if we replace $2^\omega$ by $\omega$ (for a survey on related topics see \cite{unisur}). 
Moreover by the results of
I.I. Parovichenko \cite{parovicenko} and A. S. Esenin-Volpin \cite{volpin},  
GCH implies that all the above statements are true.
The consistency of the failure has been proved
by A. Dow and K. P. Hart for Boolean algebras and compact spaces (see \cite{dowhart})
and by S. Shelah and A. Usvyatsov for isometric embeddings (see \cite{shelahusvyatsov}).
Actually one can even prove the nonexistence
of a universal Banach space in $\mathcal B_{2^\omega}$
for  isomorphic embeddings (see \cite{universal, ug}).  
 Whether the equivalence of all the three statements above  can be proved
without  additional assumptions
  is the main question
that motivates our research in this paper. Not knowing how to attack it, an 
intermediate question is whether an isometrically universal 
Banach space of the form $C(K_A)$ must be induced by a universal Boolean algebra $A$,
as is known that the existence of an isometrically universal Banach space
implies the existence of one of the form $C(K_A)$ (Fact 1.1. of \cite{universal}).
This is the case on the countable level: one can see that
if a $C(K_A)$ is separable and isometrically universal\footnote{Actually, $C(K_A)$ being isomorphically universal suffices.}
for separable Banach spaces, then $K_A$ cannot be scattered and so
$A$ must contain an infinite free algebra, meaning that $A$ is a universal algebra
among countable algebras and $K_A$ is a universal compact space
among metrizable compact spaces. The main purpose of this paper is to prove that 
the situation is different on the uncountable level:

\begin{theorem} For any uncountable $\kappa$, if there exists an isometrically universal
Banach space for $\mathcal B_\kappa$, then there is also
such a space of the form $C(K_A)$ for some Boolean algebra $A$
where $A$ does not contain well-ordered uncountable chains. In particular,
$A$ is not a universal Boolean algebra for $\mathcal A_\kappa$ nor $K_A$ is universal for $\mathcal C_\kappa$.
\end{theorem}
\begin{proof}
Suppose  that $X$ is an isometrically universal Banach space for $\mathcal B_\kappa$. By the above discussion,
by isometrically embedding $X$ into $C(B_{X^*})$ and taking a totally disconnected continuous preimage of $B_{X^*}$, we may assume that $X$ is of the form $C(K_A)$
for some Boolean algebra $A$. Then, Proposition \ref{mainprop} produces a Boolean algebra
$B$ of the same cardinality $\kappa$ such that $C(K_A)$ isometrically embeds into $C(K_B)$ 
and hence $C(K_B)$ is also isometrically universal for the same class of Banach spaces $\mathcal B_\kappa$.
However, still by Proposition \ref{mainprop}, the Boolean algebra generated by a well-ordered chain of type $\omega_1$
cannot be embedded into $B$. By the Stone duality, $K_B$ cannot be continuously mapped onto 
$[0,\omega_1]$ with the order topology. Therefore, $K_B$ is not universal for $\mathcal A_\kappa$ and $K_B$ is not universal for $\mathcal C_\kappa$.
\end{proof}

In the literature there are some results assuming the axiom OCA and concerning similar behaviour
of the Boolean algebra $\wp(\N)/Fin$ and 
 the Banach space $\ell_\infty/c_0\equiv C(K_{\wp(\N)/Fin})$.
 In fact,
there is at present no known example showing that $\ell_\infty/c_0$
is not isometrically universal under OCA while there are
many examples showing that $\wp(\N)/Fin$ is not universal (\cite{almostno, dowhartmeasure,
quotients}).
For example, A. Dow and K. P. Hart have proved in \cite{dowhartmeasure} that
OCA implies that the measure algebra $M$ does not embed into $\wp(\N)/Fin$,
while it is well known that $C(K_{M})\equiv L_\infty$ isometrically embeds into
$C(K_{\wp(\N)/Fin})\equiv\ell_\infty/c_0$ (see e.g.,  p. 304 of \cite{amirarbel}).
At the same time, while well-known that
the algebra of clopen subsets of the cozero set in $\N^*$ does
not embed into $\wp(\N)/Fin$ under OCA (see \cite{quotients}),
its is still unknown if the corresponding
Banach space $\ell_\infty(\ell_\infty/c_0)$
 can be embedded under OCA into $\ell_\infty/c_0$ (cf. \cite{sums, dowpfa}).


Our paper contains basically one construction: given a Boolean algebra
$A$, we construct another Boolean algebra $B$ such that $C(K_A)$ embeds isometrically into
$C(K_B)$ but $B$ does not contain well-ordered uncountable chains (Proposition \ref{mainprop}).
Thus, when $A$ does contain uncountable well-ordered chains as, for example, in the cases
of Boolean algebras of the clopen subsets of $[0,\omega_1]$ or $\N^*$, we cannot have
a Boolean embedding of $A$ into $B$. $B$ is the algebra of clopen sets of some (standard)
totally disconnected preimage of the dual ball $B_{C(K)^*}$ of $C(K_A)$ with the 
weak$^*$ topology. In Section 2 we show that if it had an uncountable well-ordered chain
of clopen sets we would have a chain of some sets in the ball $B_{C(K)^*}$ and this
cannot occur as is shown in Section 3.

It would be interesting to know if objects other than uncountable well-ordered chains can be used in the
above argument.
Or conversely, for which Boolean algebras $A$, the existence of an isometric embedding
$C(K_A)$ into $C(K_B)$ implies the existence of a Boolean embedding of $A$ into $B$.
By Theorem 12.30 (ii) of \cite{fabian} these cannot be uncountable antichains,
i.e., the algebra $A=FinCofin(\kappa)$ for any $\kappa$ has the above property for any algebra $B$.
Also, if $A$ has an independent family of cardinality $\kappa$ and $C(K_A)$ isometrically
embeds into $C(K_B)$, by the Holszty\'nski Theorem (\cite{holsztynski}) we have a closed set
$F\subseteq K_B$ which maps onto $\{0,1\}^\kappa$. This map can be always extended
to a totally disconnected superspace and in particular $K_B$ maps onto $\{0,1\}^\kappa$,
and so $B$ contains an independent family of cardinality $\kappa$.
But we do not know if an isometric embedding of $C(K_{\wp(\N)})\equiv\ell_\infty$
into $C(K_B)$ implies the existence of an isomorphic copy of $\wp(\N)$ in $B$.  

Terminology should be standard, concerning the Banach spaces we follow
\cite{fabian} and for Boolean algebras we follow \cite{koppelberg}. $f[X]$, $f^{-1}[X]$ denote the image  and the preimage of $X$ under $f$, respectively. $f|X$ denotes the restriction of $f$ to $X$, $\{0,1\}^{<\N}=\bigcup_{n\in \N}\{0,1\}^n$.

\section{Chains in  totally disconnected preimages}\label{secchains}

\begin{lemma}\label{interior} Let $K$ and $L$ be compact spaces and $\psi: K\rightarrow L$ be a surjective
continuous mapping. Suppose $U\subseteq K$ is clopen. For every $x \in U$, if $\psi^{-1}[\{\psi(x)\}] \subseteq U$, then $\psi(x) \in int(\psi[U])$. Hence,
\[\psi[\{x\in U: \psi^{-1}[\{\psi(x)\}]\subseteq U\}]\subseteq int(\psi[U]).\]
\end{lemma}
\begin{proof} Note that given $x \in U$, if $\psi^{-1}[\{\psi(x)\}]\subseteq U$,
then $\psi(x)\not\in \psi[K\setminus U]$. But $L = \psi[U]\cup\psi[K\setminus U]$,
 and so
$L\setminus \psi[K\setminus U]$ is an open set included in $\psi[U]$,  so $\psi(x)\in int(\psi[U])$. 
\end{proof}

The previous lemma will be applied to a restriction of the function $\phi_I$ defined as follows. Let $\phi: \{0,1\}^\N \rightarrow [0,1]$ be given by
$$\phi(x)=\sum_{n=1}^\infty \frac{x(n)}{2^n}.$$
Given any set of indices $I$, we  consider 
$\phi_I: ( \{0,1\}^\N)^I \rightarrow [0,1]^I$ which is
 defined  coordinatewise by 
\[\phi_I(x)(i) 
= \phi(x(i))\]
 for $x \in ( \{0,1\}^\N)^I$ and $i \in I$.
Since $\phi_I$ is defined
coordinatewise we immediately obtain the following:

\begin{lemma}\label{combination} Suppose $x, x', x''\in  ( \{0,1\}^\N)^I$ are such that $\phi_I(x)=\phi_I(x')$
and that  that for each $i\in I$ either $x''(i)=x(i)$ or $x''(i)=x'(i)$. Then
$\phi_I(x'')=\phi_I(x)=\phi_I(x')$.
\end{lemma}

We will consider the standard basis of clopen sets of $ \{0,1\}^\N$, i.e., the sets of the form
$$[s]=\{x\in  \{0,1\}^\N: s\subseteq x\}$$
for  $s\in  \{0,1\}^{<\N}$.  By the definition of the product topology, the sets of the form
$$U(i, [s]) = \{x \in (\{0,1\}^\N)^I: x(i) \in [s]\},$$
for $i \in I$ and $s\in  \{0,1\}^{<\N}$ form a topological subbasis 
 for $(\{0,1\}^\N)^{I}$ which consists of clopen sets. 
Note that
$$\phi[[s]]=\Big[\sum_{1\leq n\leq |s|}\frac{x(n)}{2^n},\ \  
\ \sum_{1\leq n\leq |s|}\frac{x(n)}{2^n}+\frac{1}{2^{|s|}}\Big]
\leqno (*)$$
for any $s\in \{0,1\}^{<\N}$. Hence, $\phi$ sends the standard basic clopen sets onto
closed subintervals of $[0,1]$, in particular onto convex sets.

\begin{definition}
For a subspace $X$ of $(\{0,1\}^\N)^I$, $n\in \N$ and a subset $J\subseteq I$
 we will say that $Y\subseteq X$ $n$-depends on $J\subseteq I$
in $X$ if and only if whenever $x, y\in X$ and $x(i)|n=y(i)|n$
for  each $i\in J$, then
\[x\in Y \ \Leftrightarrow\  y\in Y.\]
\end{definition}

It is immediate that $U(i,[s])$ $n$-depends on $\{i\}$ in $(\{0,1\}^\N)^I$ where $n=|s|$
and so, $U(i,[s]) \cap L$ $n$-depends on $\{i\}$ in any $L\subseteq (\{0,1\}^\N)^I$. Since clopen
sets of compact Hausdorff spaces are Boolean combinations of finitely many sets from any clopen subbasis, it follows that
any clopen subset of any closed $L\subseteq (\{0,1\}^\N)^I$
$n$-depends on $J\subseteq I$ in $L$, for $n$ being any natural number larger than $|s|$ for all $s$ which appear in the finite Boolean combination of sets of the form $U(i,[s])$ and $J$ is the set of all $i$ which appear in such a Boolean combination. 
We are now ready
for the main result of this section.

\begin{proposition}\label{chain} Let $I$ be a set and 
let $K$ be a closed convex subspace of $[0,1]^I$
 such that no closed convex subspace $F$ of $K$ has an uncountable well-ordered chain of open (in $F$) sets $(V_\alpha)_{\alpha<\omega_1}$ satisfying
 $\overline{V}_\alpha \subseteq V_\beta$ for $\alpha<\beta<\omega_1$. Then $L = {\phi_I}^{-1}[K]$ has no uncountable well-ordered chain of clopen sets $(U_\alpha)_{\alpha<\omega_1}$.
\end{proposition}
\begin{proof} Suppose that 
 $(U_\alpha\cap L)_{\alpha \in \omega_1}$ is a sequence 
of clopen subsets of $L$ where $U_\alpha$'s
are clopen subsets of $(\{0,1\}^\N)^I$. As we noticed, it follows that each $U_\alpha \cap L$ 
$n_\alpha$-depends on some finite set $J_\alpha\subseteq I$ in $(\{0,1\}^\N)^I$,
for some $n_\alpha$.
Using the $\Delta$-system lemma (see \cite{jech}) and the fact that $\{0,1\}^{<\N}$ and $\N$ are countable, 
we may assume that $(J_\alpha)_{\alpha < \omega_1}$ is a $\Delta$-system with root $\Delta$
and each $U_\alpha$ $n$-depends on $J_\alpha$ in $(\{0,1\}^\N)^I$ for a fixed $n\in \N$ and all $\alpha<\omega_1$.

For each $f\in (\{0,1\}^n)^\Delta$, consider 
\[U(f)=\bigcap_{i\in \Delta} U(i, [f(i)]).\]
 $(\{0,1\}^{\N})^I$ is the disjoint union of the family of clopen sets $\{U(f): f\in (\{0,1\}^n)^\Delta\}$.
It follows that there is $f_0\in (\{0,1\}^n)^\Delta$ such that 
$(U_\alpha\cap L\cap U(f_0))_{\alpha<\omega_1}$ forms an uncountable sequence. By
going to a subsequence we may assume that all elements $U_\alpha\cap L\cap U(f_0)$ are distinct.
Consider $L'=L\cap U(f_0)$, $F=\phi_I[L']$ and $\psi=\phi_I| L'$ from $L'$ onto $F$ and put
\[V_\alpha=int_{F}(\psi [U_\alpha\cap L']).\]
Note that $F=\phi_I[\phi_I^{-1}[K] \cap U(f_0)] = \phi_I[U(f_0)] \cap K$, which is 
convex as the intersection of two convex sets.
Secondly, note that $U_\alpha\cap L'$ $n$-depends on $J_\alpha\setminus\Delta$ in $L'$, for each $\alpha<\omega_1$.
Indeed, whenever $x, y\in L'$ 
we have that $x(i)|n=f_0(i)
=y(i)|n$ for all $i\in\Delta$ and so, whenever we have additionally that
 $x(i)|n=y(i)|n$ for all $i\in J_\alpha\setminus\Delta$,
we may use the fact that $U_\alpha$ $n$-depends on $J_\alpha$ in $(\{0,1\}^\N)^{I}$.

By the hypothesis on convex sets in $K$ applied to $F$, there are $\alpha<\beta<\omega_1$
such that
$\overline{V_\alpha}\not\subseteq V_\beta$. Aiming at a 
contradiction, let us assume that $(U_\alpha \cap L)_{\alpha < \omega_1}$ is a well-ordered chain and hence, $U_\alpha\cap L'\subseteq U_\beta\cap L'$. Then, since $\psi[U_\alpha\cap L']$ is a closed set,
$\overline{V_\alpha}\subseteq \psi[U_\alpha\cap L']$ and
we conclude that there is
\[y\in \psi [U_\alpha\cap L']\setminus int_F(\psi [U_\beta\cap L'])
\subseteq \psi [U_\beta\cap L']\setminus int_F(\psi [U_\beta\cap L']).\] 
Let $x\in U_\alpha\cap L'$ be such that $\psi(x)=y$.
 Lemma \ref{interior} gives that $\psi^{-1}[\{y\}]\not\subseteq U_\beta\cap L'$ and so
there is $x'\in L'$ such that  $x'\not \in U_\beta\cap L'$ but $\psi(x')=y$. 
Now we will combine $x$ and $x'$ following  Lemma \ref{combination} and will obtain
a contradiction with the hypothesis that $U_\alpha\cap L'\subseteq U_\beta\cap L'$.
Define
$$x''(\xi) = \left\{ \begin{array}{ll}
     x'(\xi) & \text{if } \xi \not \in J_\beta \setminus \Delta\\
     x(\xi) & \text{otherwise.}
         \end{array}
                                         \right.$$ 
By Lemma \ref{combination} the point $x''$ is in $L$.
Note that $x''\in U(f_0)$ as $x''|\Delta=x|\Delta$ i.e., $x''\in L\cap U(f_0)=L'$. Also,
$x''|J_\alpha=x|J_\alpha$ and so $x''\in U_\alpha$ 
since $U_\alpha\cap L'$ $n$-depends on $J_\alpha$ in $L'$ and $x \in U_\alpha$. On the other
hand, $x''|(J_\beta\setminus\Delta)=x'|(J_\beta\setminus\Delta)$, so
$x''\not \in U_\beta$ as $U_\beta\cap L'$ $n$-depends on $J_\beta\setminus\Delta$ in $L'$ and $x' \notin U_\beta$.
This shows that $U_\alpha \cap L' \not\subseteq U_\beta \cap L'$, contradicting our hypothesis. Hence, $(U_\alpha \cap L)_{\alpha<\omega_1}$
is not a well-ordered chain and this completes the proof of the proposition.
\end{proof}

\section{Well-ordered chains in the dual ball}

\begin{proposition}\label{ballBanach}
If $F$ is any closed convex subspace of the dual unit ball of a Banach space endowed with the weak$^*$ topology, then $F$ does not have an uncountable well-ordered chain of open sets $(V_\alpha)_{\alpha < \omega_1}$ such that 
$\overline{V_\alpha} \subseteq V_\beta$ for any $\alpha<\beta<\omega_1$. 
\end{proposition}
\begin{proof}
Let $X$ be a Banach space, $B_{X^*}$ its dual unit ball endowed with the weak$^*$ 
topology and $F$ a closed convex subspace of $B_{X^*}$. Suppose $(V_\alpha)_{\alpha \in \omega_1}$ is a well-ordered chain of
open sets of $F$ such that 
$\overline{V_\alpha} \subseteq V_\beta$ for any $\alpha<\beta<\omega_1$. Put $V = \bigcup_{\alpha \in \omega_1} V_\alpha$. Since $F$ is a closed subspace of $B_{X^*}$ and $B_{X^*}$ is weakly$^*$ compact, then $F \setminus V$ is a nonempty closed set of $F$. It is as well 
a weakly closed set, since the weak topology is finer than the weak$^*$ topology. 
It cannot be a weakly open set in $F$, because $F$ is convex, and since $X$ endowed with the weak topology is a topological vector space, $F$ is weakly connected. Then, there is 
$x \in \overline{B_{X^*} \setminus V}^{w} \cap \overline{V}^{w}$. 

Now, by Kaplansky's theorem (e.g., 4.49 of \cite{fabian}), every Banach 
space has countable tightness in its weak topology. So, there is a countable set 
$D \subseteq V$ such that $x \in \overline{D}^w$. Since $D$ is countable, there
is $\gamma\in \omega_1$ 
such that $D \subseteq V_\gamma$, which implies that 
$x \in \overline{V}^w_\gamma \subseteq \overline{V}_\gamma^{w^*} \subseteq V_{\gamma+1} \subseteq V$, 
contradicting the fact that $x \in  \overline{F \setminus V}^{w}$.
\end{proof}

The hypothesis of $F$ being convex is crucial, as any compact Hausdorff space $K$
can be homeomorphically embedded in the dual ball of $C(K)$ with the weak$^*$
topology by associating
the Dirac $\delta_x$ to each $x\in K$.
In the context of the above result it is also worthy to mention the following:

\begin{proposition} 
Note that if $X$ is a Banach space of density $\kappa$, then there is a well-ordered
increasing chain $(U_\xi)_{\xi<\kappa}$ of open sets  in the dual unit ball $B_{X^*}$ of $X$
with the weak$^*$ topology. 
\end{proposition}
\begin{proof}
Using the Hahn-Banach theorem,
construct by transfinite induction a sequence $(x_\xi, x^*_\xi)_{\xi<\kappa} \subseteq X \times B_{X^*}$
such that $x^*_\xi(x_\eta)=0$ for $\eta<\xi<\kappa$ and $x^*_\xi(x_\xi)=1$ for
$\xi<\kappa$. Then 
$$U_\xi=\bigcup_{\eta<\xi}\{x^*\in B_{X^*}: x^*(x_\eta)\not=0\}$$
is as required.
\end{proof}

\begin{lemma}\label{lema-convex}
Given a Banach space $X$ and a dense
subset $D$ of its unit ball, the natural restriction mapping $f: B_{X^*} \rightarrow [-1,1]^{D}$ defined 
by $f(x^*)=x^*|_{D}$ is a homeomorphism onto its image with respect to
the weak$^*$ and the product topologies, with the property
 that $F \subseteq B_{X^*}$ is convex if and only if $f(F)$ is convex. 
\end{lemma}
\begin{proof} The preimages of standard basic open sets in the product are
weakly$^*$ open, so that $f$ is continuous. As two distinct functionals must 
differ on an element of the unit ball we see that
 $f$ is a homeomorphism onto its image. 
$\Phi: X^*\rightarrow \R^D$ given by $\Phi(x^*)(d)=x^*(d)$ is linear and one-to-one
and hence its inverse is linear and one-to-one. Both mappings preserve convexity.
The lemma follows as $f=\Phi|B_{X^*}$.
\end{proof}

\begin{proposition}\label{pre-image-Banach}
The dual unit ball of every Banach space endowed with the weak$^*$ topology has a continuous preimage which is compact, totally disconnected  of same weight and with no uncountable well-ordered chain of clopen sets.
\end{proposition}
\begin{proof} Let $X$ be a Banach space of density $\kappa$.
Let $D\subseteq X$ be a dense subset of the unit ball of cardinality $\kappa$.
Let $f: B_{X^*} \rightarrow [-1,1]^{D}$ be defined by $f(x^*)=x^*|_{B_{X}}$ 
as in Lemma \ref{lema-convex}, let $g$ be the linear 
order-preserving homeomorphism from $[-1,1]$ onto $[0,1]$
 and let ${g}^D: [-1,1]^{D} \rightarrow [0,1]^{D}$ be defined coordinatewise by
 ${g}^D(x)(d) = g(x(d))$ for any $d\in D$ and $x\in [-1,1]^D$. Clearly
 ${g}^D$ is a homeomorphism such that $F \subseteq [-1,1]^{D}$
 is convex if and only if $g^D[F]$ is convex.
 Then, $h=g \circ f: B_{X^*} \rightarrow [0,1]^{D}$
 is a homeomorphism onto its image such that $F \subseteq B_{X^*}$ 
is convex if and only if $h[F]$ is convex. 

Let $K = h[B_{X^*}] \subseteq [0,1]^{D}$. If $F$ is a convex subset of $K$, then $h^{-1}[F]$ is a convex subset of $B_{X^*}$ and Proposition \ref{ballBanach} guarantees that it 
contains no uncountable well-ordered chain of open sets $(V_\alpha)_{\alpha<\omega_1}$ such that $\overline{V}_\alpha^F \subseteq V_\beta$ for $\alpha < \beta <\omega_1$. Since $h$ is a homeomorphism onto $K$, we get that $F$ has no such  chain of open sets either. 

Finally, since $K$ satisfies the hypotheses of Proposition \ref{chain},
 we get that $L = {(\phi_D)}^{-1}[K]$ has no uncountable well-ordered chain of clopen sets.
Since the weight of $K$ cannot be bigger than $D$, this concludes the proof. 
\end{proof}

\begin{proposition}\label{mainprop} Suppose that
 $A$ is a Boolean algebra.  There is a Boolean algebra ${B}$ 
of same cardinality as ${A}$ but without uncountable
well-ordered
chains   such that the Banach space $C(K_{B})$ contains an isometric copy of $C(K_{A})$.
\end{proposition}
\begin{proof}
Let $X= C(K_{A})$ and by Proposition \ref{pre-image-Banach}, the dual unit ball $B_{X^*}$ has a continuous preimage $L$ which is compact, totally disconnected, of the same weight as $B_{C(K_{A})^*}$ 
and with no uncountable well-ordered chain of clopen sets. Hence, ${B}= Clop (L)$ is a Boolean algebra of same cardinality as ${A}$ which has no uncountable well-ordered chains and, therefore, has no isomorphic copy of ${A}$. But $K_{B}$ is homeomorphic to $L$, so that $C(K_{B})$ is isometric to $C(L)$, which contains an isometric copy of $C(B_{X^*})$, which in turn contains an isometric copy of $X = C(K_{A})$.
\end{proof}

\bibliographystyle{amsplain}

\end{document}